\documentclass[12pt,a4paper]{article}

\usepackage{amsmath,amssymb,amsfonts,mathtools,bm}
\usepackage{geometry}
\usepackage{hyperref}
\hypersetup{colorlinks=true, linkcolor=blue, citecolor=blue}
\usepackage[nameinlink,capitalise]{cleveref}

\newcommand{\Sph}{\mathbb S}
\newcommand{\RP}{\mathbb{RP}}
\newcommand{\R}{\mathbb R}
\newcommand{\N}{\mathbb N}
\newcommand{\K}{\mathcal K}
\newcommand{\Harm}{\mathcal H}
\newcommand{\cB}{\mathcal B}

\newcommand{\Id}{\operatorname{Id}}
\newcommand{\sym}{\operatorname{sym}}

\usepackage{amsthm}
\theoremstyle{plain}
\newtheorem{theorem}{Theorem}[section]
\newtheorem{proposition}[theorem]{Proposition}

\newtheorem{corollary}[theorem]{Corollary}

\theoremstyle{definition}
\newtheorem{definition}[theorem]{Definition}
\newtheorem{example}[theorem]{Example}
\newtheorem{remark}[theorem]{Remark}

\begin{document}

% ============================================================
% 标题（手动排版）
% ============================================================
\begin{center}
{\LARGE \textbf{Exact Finite Integral Geometry:}}\\[0.2cm]
{\LARGE \textbf{Directional Kernels and Spherical Designs}}\\[0.6cm]
{\large Congpei An}\\[0.2cm]
{\normalsize School of Mathematics, Guangxi University, Nanning, China}\\[0.1cm]
{\normalsize \href{mailto:andbachcp@gmail.com}{andbachcp@gmail.com}}\\[0.5cm]
\end{center}

% ============================================================
% 摘要
% ============================================================
\begin{abstract}
We study when an invariant directional average from integral geometry can be replaced exactly by finitely many directions. For an even continuous kernel $\psi:[-1,1]\to\R$, let $T_\psi$ be the associated zonal convolution operator on the sphere and let $\cB_\psi K$ be the corresponding surface-area observable of a convex body $K\subset\R^d$. Our first result is a sharp norm identity: the worst relative error over all convex bodies equals one half of the $L^\infty$ norm of the potential discrepancy $T_\psi(\mu-\sigma)$, for every normalized signed directional measure $\mu$. Hence universal exactness is equivalent to $\mu-\sigma\in\ker T_\psi$. Funk--Hecke diagonalization then gives a complete spectral criterion: one must annihilate exactly the spherical harmonic degrees on which the multiplier of $T_\psi$ is nonzero. This yields kernel-adapted designs and positive finite exact rules for finite active spectrum. For $\psi_p(s)=|s|^p$ we obtain a complete classification: even powers give precisely weighted real projective designs, whereas non-even powers admit no finite signed atomic rule that is exact for every convex body. For the classical Cauchy kernel $|s|$, exactness fails but spherical $t$-designs give a uniform $O(t^{-1})$ relative surface-area error. We also derive exact projection-moment identities for rectifiable submanifolds.
\end{abstract}

\vspace{0.3cm}
\hrule
\vspace{0.3cm}

% ============================================================
% 正文开始
% ============================================================
\section{Introduction}

A characteristic feature of integral geometry is the recovery of intrinsic geometric information from invariant averages over extrinsic configurations. Crofton formulas average intersection counts, Cauchy formulas average projection data, and kinematic formulas average intersection invariants over relative positions. Santal\'o's classical treatment makes this viewpoint systematic \cite{Santalo}. In geometric tomography, projection and section data are encoded by rotation-covariant integral transforms; see Gardner \cite{Gardner}. Multiplier representations of broad classes of such transforms are classical, notably in Kiderlen's framework \cite{Kiderlen} and in the harmonic analysis of cosine-type transforms \cite{GYY}.

The problem considered here has a simple harmonic-analytic formulation but a genuinely geometric quantifier. We ask:
\begin{quote}
\emph{When does a directional integral-geometric identity admit an exact finite realization for every convex body?}
\end{quote}
At the analytic level, the relevant object is a zonal convolution operator $T_\psi$ on the sphere. At the geometric level, the test class consists of all surface area measures of full-dimensional convex bodies. The key observation is that these surface area measures form a norming family for even continuous functions. This turns the universal convex-body problem into an exact operator-norm identity, after which Funk--Hecke diagonalization identifies the nullspace spectrally. Thus the answer is governed not by a single global design degree, but by the harmonic spectrum actually seen by the geometric kernel.

Let $\sigma$ be normalized rotation-invariant measure on $\Sph^{d-1}$ and let $S_K$ denote the surface area measure of a convex body $K\subset\R^d$. For an even continuous kernel $\psi:[-1,1]\to\R$, define
\begin{equation}\label{eq:kernel-observable-intro}
 \cB_\psi K(u)=\frac12\int_{\Sph^{d-1}}\psi(\langle u,v\rangle)\,dS_K(v).
\end{equation}
The case $\psi(s)=|s|$ is the classical Cauchy brightness function $\cB_{|\cdot|}K(u)=\operatorname{vol}_{d-1}(K|u^\perp)$. For a normalized finite signed Borel measure $\mu$ on the sphere, set
\[
 \Delta_{\psi,\mu}(K):=\int\cB_\psi K\,d\mu-\int\cB_\psi K\,d\sigma.
\]

Our main contributions are threefold.
\begin{enumerate}
\item \emph{A sharp norm identity.} We prove
\begin{equation}\label{eq:defect-intro}
 \sup_{K\in\K^d_\circ}\frac{|\Delta_{\psi,\mu}(K)|}{S(K)}
 =\frac12\|T_\psi(\mu-\sigma)\|_{L^\infty(\Sph^{d-1})}.
\end{equation}
The upper bound is immediate; the reverse inequality is the geometric step. Minkowski's existence theorem allows normalized surface area measures to approximate an antipodal point mass while preserving full-dimensional support. Equivalently, normalized surface area measures are a norming family for the even subspace of $C(\Sph^{d-1})$. Thus the entire worst-case convex-body defect is exactly the $L^\infty$ discrepancy of a zonal convolution potential.

\item \emph{Identification of the nullspace of the geometric averaging problem.} Let $\lambda_\ell(\psi)$ denote the Funk--Hecke multiplier on $\Harm_\ell$ and set $\Lambda_\psi=\{\ell\ge1:\lambda_\ell(\psi)\ne0\}$. From \eqref{eq:defect-intro}, universal exactness is equivalent to $\mu-\sigma\in\ker T_\psi$. Funk--Hecke diagonalization yields the precise condition
\begin{equation}\label{eq:criterion-intro}
 \Delta_{\psi,\mu}(K)=0\ \text{for every }K\in\K^d_\circ
 \quad\Longleftrightarrow\quad
 \int Y\,d\mu=0\ \text{for all }Y\in\Harm_\ell,\ \ell\in\Lambda_\psi.
\end{equation}
Hence the kernel itself selects the exact moment conditions. This leads to kernel-adapted harmonic-index designs, and finite active spectrum automatically yields a positive finite exact rule by a Carath\'eodory argument.

\item \emph{A model classification and the transition beyond exactness.} For $\psi_p(s)=|s|^p$ we compute
\begin{equation}\label{eq:power-mult-intro}
 \lambda_{2n}(|\cdot|^p)
 =c_{d,p}\frac{(-1)^n(-p/2)_n}{((d+p)/2)_n},
 \qquad \lambda_{2n+1}(|\cdot|^p)=0,
\end{equation}
where $c_{d,p}=\Gamma(d/2)\Gamma((p+1)/2)/(\sqrt\pi\,\Gamma((d+p)/2))$. If $p=2m$, the active spectrum is $\{2,4,\ldots,2m\}$ and exactness is precisely a weighted real projective $m$-design condition. If $p>0$ is not even, every positive even degree is active and no finite signed atomic rule is universally exact. For the classical Cauchy kernel, spherical $t$-designs nevertheless give a uniform error bounded by $E_t(|\cdot|)_\infty S(K)$, hence a relative $O(t^{-1})$ estimate. We also obtain exact projection-moment identities for rectifiable submanifolds.
\end{enumerate}

The relation between even powers, projective cubature and finite-dimensional $\ell_p$ embeddings is classical \cite{LyubichVaserstein,Lyubich2009,Lindblad}; likewise, harmonic-index designs are established objects \cite{BOT,ZhuEtAl}. Our contribution is not the sphere--projective correspondence or the even-exponent phenomenon in isolation. It is the universal convex-body defect identity \eqref{eq:defect-intro}, the resulting kernel-by-kernel characterization \eqref{eq:criterion-intro}, and the interpretation of the active harmonic spectrum as the exact finite design requirement for a directional integral-geometric observable. This perspective is complementary to multiplier theory in geometric tomography \cite{Gardner,Kiderlen,GYY}, to $L_p$ cosine-transform methods \cite{LiXiZhang}, and to numerical discretization of integral-geometric formulas on groups and Grassmannians \cite{Pausinger,PausingerSvane}.

The paper is intentionally restricted to directional geometry on $\Sph^{d-1}/\{\pm1\}\cong\RP^{d-1}$. Genuine Crofton and kinematic formulas on Grassmannians and motion groups require a broader representation-theoretic setting. The directional case already exhibits the central mechanism in a complete form and provides the first layer of the broader program that we call \emph{Exact Finite Integral Geometry}.

\section{Harmonic and convex-geometric preliminaries}

Throughout, $d\ge2$.  Let $\sigma$ denote normalized Haar probability measure on $\Sph^{d-1}$.  We write $\Harm_\ell=\Harm_\ell(\R^d)$ for the restrictions to the sphere of real homogeneous harmonic polynomials of degree $\ell$.  Its dimension is
\begin{equation}\label{eq:dim-harm}
 h_\ell^{(d)}
 =\binom{d+\ell-1}{\ell}-\binom{d+\ell-3}{\ell-2},
\end{equation}
with the second binomial interpreted as zero for $\ell<2$.

Every $f\in L^2(\Sph^{d-1})$ has an orthogonal spherical harmonic expansion
\[
 f\sim\sum_{\ell=0}^\infty f_\ell,
 \qquad f_\ell\in\Harm_\ell.
\]
If $f$ is even, only even degrees occur.

\paragraph{Funk--Hecke multipliers.}

For $\psi\in C([-1,1])$ define the zonal convolution operator
\begin{equation}\label{eq:Tpsi}
 (T_\psi\nu)(v)
 :=\int_{\Sph^{d-1}}\psi(\langle u,v\rangle)\,d\nu(u)
\end{equation}
for any finite signed Borel measure $\nu$ on $\Sph^{d-1}$.  When $\nu=f\sigma$, we also write $T_\psi f$.

The Funk--Hecke theorem says that each $\Harm_\ell$ is an eigenspace.  We normalize the multiplier $\lambda_\ell(\psi)$ by
\begin{equation}\label{eq:FH}
 \int_{\Sph^{d-1}}\psi(\langle u,v\rangle)Y(v)\,d\sigma(v)
 =\lambda_\ell(\psi)Y(u),
 \qquad Y\in\Harm_\ell.
\end{equation}
In particular,
\begin{equation}\label{eq:lambda0}
 \lambda_0(\psi)
 =\int_{\Sph^{d-1}}\psi(\langle u,e\rangle)\,d\sigma(u),
\end{equation}
independent of $e\in\Sph^{d-1}$.

\begin{definition}[Active spectrum]\label{def:active}
For $\psi\in C([-1,1])$, define
\[
 \Lambda_\psi
 :=\{\ell\in\N:\lambda_\ell(\psi)\ne0\}.
\]
If $\psi$ is even, then $\Lambda_\psi\subset2\N$.
\end{definition}

\paragraph{Surface area measures.}

Let $\K^d$ be the class of convex bodies in $\R^d$ and $\K^d_\circ$ the subclass with nonempty interior.  The surface area measure $S_K$ of $K\in\K^d_\circ$ is the pushforward of $(d-1)$-dimensional Hausdorff measure on $\partial K$ under the outer Gauss map.  Its total mass is the surface area
\[
 S(K):=S_K(\Sph^{d-1}).
\]
We use the following classical form of Minkowski's existence theorem; see \cite[Sec.~8.2]{Schneider}.

\begin{theorem}[Minkowski existence theorem]\label{thm:minkowski}
Let $\eta$ be a finite positive Borel measure on $\Sph^{d-1}$ that is not concentrated on any great subsphere and satisfies
\[
 \int_{\Sph^{d-1}}u\,d\eta(u)=0.
\]
Then there exists a convex body $K\in\K^d_\circ$, unique up to translation, such that $S_K=\eta$.
\end{theorem}

For our purposes the key consequence is a separation principle.

\paragraph{Spherical, harmonic-index, and projective designs.}

A normalized finite measure
\[
 \mu=\sum_{j=1}^Nw_j\delta_{x_j},\qquad \sum_jw_j=1,
\]
is a weighted spherical $t$-design if it agrees with $\sigma$ on all spherical polynomials of degree at most $t$.  Equivalently,
\begin{equation}\label{eq:spherical-design-harm}
 \int Y\,d\mu=0,
 \qquad Y\in\Harm_\ell,
 \quad 1\le\ell\le t.
\end{equation}
For equal weights this is the standard spherical design condition \cite{DGS,Seidel}.

More generally, if $T\subset\N$, we say that $\mu$ is a \emph{weighted harmonic-index $T$ design} if
\begin{equation}\label{eq:Tdesign}
 \int Y\,d\mu=0
 \qquad \forall\,Y\in\Harm_\ell,\ \ell\in T.
\end{equation}
The equal-weight notion for one index was introduced in \cite{BOT}; multiple index sets are studied in \cite{ZhuEtAl}.

When only even degrees are relevant, the geometry is naturally projective.  The quotient map
\[
 \pi:\Sph^{d-1}\to\RP^{d-1},\qquad \pi(u)=[u],
\]
identifies $u$ and $-u$.  Polynomial functions of projective degree at most $m$ pull back to even spherical polynomials of degree at most $2m$.  Thus annihilation of $\Harm_2,\Harm_4,\ldots,\Harm_{2m}$ is the weighted real projective $m$-design condition; see, for example, \cite{Lyubich2009,Lindblad}.

\section{A sharp norm identity and the nullspace of the zonal operator}

Let $\psi\in C([-1,1])$ be even.  For $K\in\K^d_\circ$, define
\begin{equation}\label{eq:Bpsi}
 \cB_\psi K(u)
 :=\frac12\int_{\Sph^{d-1}}\psi(\langle u,v\rangle)\,dS_K(v).
\end{equation}
Because $\psi$ is even, $\cB_\psi K(u)=\cB_\psi K(-u)$, so the directional variable is intrinsically projective.

\begin{proposition}[Continuous invariant average]\label{prop:continuous-average}
For every $K\in\K^d_\circ$,
\begin{equation}\label{eq:continuous-average}
 \int_{\Sph^{d-1}}\cB_\psi K(u)\,d\sigma(u)
 =\frac{\lambda_0(\psi)}2\,S(K).
\end{equation}
\end{proposition}

\begin{proof}
By Fubini and rotational invariance,
\[
 \begin{aligned}
 \int \cB_\psi K(u)\,d\sigma(u)
 &=\frac12\int_{\Sph^{d-1}}
 \left[\int_{\Sph^{d-1}}\psi(\langle u,v\rangle)\,d\sigma(u)\right]dS_K(v)\\
 &=\frac{\lambda_0(\psi)}2S(K).
 \end{aligned}
\]
\end{proof}

We first identify the exact worst-case error over the whole class of convex bodies.

\begin{theorem}[Sharp Defect Theorem]\label{thm:sharp-defect}
Let $\psi\in C([-1,1])$ be even and let $\mu$ be a finite signed Borel measure on $\Sph^{d-1}$ with $\mu(\Sph^{d-1})=1$. Then
\begin{equation}\label{eq:sharp-defect}
 \boxed{
 \sup_{K\in\K^d_\circ}
 \frac{1}{S(K)}
 \left|\int\cB_\psi K\,d\mu-\int\cB_\psi K\,d\sigma\right|
 =\frac12\|T_\psi(\mu-\sigma)\|_\infty .
 }
\end{equation}
\end{theorem}

\begin{proof}
Put
\[
 h(v):=T_\psi(\mu-\sigma)(v).
\]
Because $\psi$ is even, $h$ is continuous and even. Since $\mu$ and $\sigma$ have the same total mass,
\[
 \int_{\Sph^{d-1}}h\,d\sigma=0.
\]
Fubini gives
\begin{equation}\label{eq:defect-h}
 \int\cB_\psi K\,d\mu-\int\cB_\psi K\,d\sigma
 =\frac12\int_{\Sph^{d-1}}h(v)\,dS_K(v).
\end{equation}
Hence the left side of \eqref{eq:sharp-defect} is at most $\frac12\|h\|_\infty$.

For the reverse inequality, choose $v_0\in\Sph^{d-1}$ with $|h(v_0)|=\|h\|_\infty$. For $0<\varepsilon<1$ define the even probability measure
\[
 \eta_\varepsilon
 =(1-\varepsilon)\frac{\delta_{v_0}+\delta_{-v_0}}2+\varepsilon\sigma.
\]
It has zero first moment and, because of the $\varepsilon\sigma$ term, is not concentrated on any great subsphere. Minkowski's existence theorem therefore yields $K_\varepsilon\in\K^d_\circ$ with $S_{K_\varepsilon}=\eta_\varepsilon$. In particular, $S(K_\varepsilon)=1$. Using the evenness of $h$ and its zero $\sigma$-mean,
\[
 \int h\,dS_{K_\varepsilon}
 =(1-\varepsilon)h(v_0).
\]
Substitution into \eqref{eq:defect-h} and letting $\varepsilon\downarrow0$ proves the reverse inequality.
\end{proof}

\begin{remark}[Operator-theoretic form]\label{rem:operator-form}
If $\nu=\mu-\sigma$, then \cref{thm:sharp-defect} says that the convex-body defect seminorm is exactly
\[
 \mathfrak D_\psi(\nu):=
 \sup_{K\in\K^d_\circ}\frac{1}{S(K)}
 \left|\int\cB_\psi K\,d\nu\right|
 =\frac12\|T_\psi\nu\|_\infty.
\]
Thus universal exactness is the nullspace condition $\nu\in\ker T_\psi$. The role of convex geometry is to identify the natural test family with the full $L^\infty$ norm on the even range of the zonal operator; the role of harmonic analysis is then to diagonalize that nullspace.
\end{remark}

Universal exactness is therefore equivalent to the vanishing of the potential discrepancy. The next theorem identifies this condition spectrally.

\begin{theorem}[Universal kernel exactness criterion]\label{thm:main-criterion}
Let $\psi\in C([-1,1])$ be even, and let $\mu$ be a finite signed Borel measure on $\Sph^{d-1}$ with $\mu(\Sph^{d-1})=1$.  The following are equivalent.
\begin{enumerate}
\item For every $K\in\K^d_\circ$,
\begin{equation}\label{eq:main-i}
 \int\cB_\psi K\,d\mu
 =\int\cB_\psi K\,d\sigma.
\end{equation}
\item The $\psi$-potential of $\mu$ is constant:
\begin{equation}\label{eq:main-ii}
 T_\psi\mu(v)=\lambda_0(\psi)
 \qquad\forall v\in\Sph^{d-1}.
\end{equation}
\item For every $\ell\ge1$ and every $Y\in\Harm_\ell$,
\begin{equation}\label{eq:main-iii}
 \lambda_\ell(\psi)\int Y\,d\mu=0.
\end{equation}
\item The measure $\mu$ annihilates exactly the active harmonic degrees:
\begin{equation}\label{eq:main-iv}
 \int Y\,d\mu=0
 \qquad\forall Y\in\Harm_\ell,
 \quad \ell\in\Lambda_\psi.
\end{equation}
\end{enumerate}
\end{theorem}

\begin{proof}
By \cref{thm:sharp-defect}, (i) is equivalent to
\[
 T_\psi(\mu-\sigma)\equiv0,
\]
which is exactly (ii), since $T_\psi\sigma\equiv\lambda_0(\psi)$.

Set $h=T_\psi\mu-\lambda_0(\psi)$. For $Y\in\Harm_\ell$ with $\ell\ge1$, Fubini and the Funk--Hecke relation \eqref{eq:FH} give
\[
 \begin{aligned}
 \int_{\Sph^{d-1}}h(v)Y(v)\,d\sigma(v)
 &=\int_{\Sph^{d-1}}
 \left[\int_{\Sph^{d-1}}\psi(\langle u,v\rangle)Y(v)\,d\sigma(v)\right]d\mu(u)\\
 &=\lambda_\ell(\psi)\int_{\Sph^{d-1}}Y(u)\,d\mu(u).
 \end{aligned}
\]
Hence (ii) implies (iii).  Conversely, if (iii) holds, every nonconstant spherical harmonic coefficient of $h$ vanishes.  Also
\[
 \int h\,d\sigma
 =\lambda_0(\psi)\bigl(\mu(\Sph^{d-1})-1\bigr)=0.
\]
Therefore $h=0$ in $L^2(\sigma)$, and continuity gives $h\equiv0$.  Finally, (iii) and (iv) are equivalent by the definition of $\Lambda_\psi$.
\end{proof}

\begin{corollary}[Finite-spectrum harmonic discrepancy]\label{cor:l2-discrepancy}
Assume that the active spectrum $\Lambda_\psi=T$ is finite. For each $\ell\in T$, let $\{Y_{\ell,r}\}_{r=1}^{h_\ell^{(d)}}$ be an orthonormal basis of $\Harm_\ell$ in $L^2(\sigma)$ and set
\[
 \mathcal E_{2,\psi}(\mu)^2
 :=\sum_{\ell\in T}|\lambda_\ell(\psi)|^2
 \sum_{r=1}^{h_\ell^{(d)}}\left|\int Y_{\ell,r}\,d\mu\right|^2,
 \qquad
 M_T:=\sum_{\ell\in T}h_\ell^{(d)}.
\]
Then
\begin{equation}\label{eq:l2-linf-defect}
 \frac12\mathcal E_{2,\psi}(\mu)
 \le
 \sup_{K\in\K^d_\circ}\frac{|\Delta_{\psi,\mu}(K)|}{S(K)}
 \le
 \frac12\sqrt{M_T}\,\mathcal E_{2,\psi}(\mu).
\end{equation}
In particular, the convex-body defect is quantitatively equivalent, up to the dimension of the active representation space, to the weighted $L^2$ harmonic discrepancy selected by the kernel.
\end{corollary}

\begin{proof}
For $h=T_\psi(\mu-\sigma)$, finite spectral support and Funk--Hecke diagonalization give
\[
 h=\sum_{\ell\in T}\lambda_\ell(\psi)
 \sum_{r=1}^{h_\ell^{(d)}}\left(\int Y_{\ell,r}\,d\mu\right)Y_{\ell,r}.
\]
Hence Parseval gives $\|h\|_2=\mathcal E_{2,\psi}(\mu)$. Since $\sigma$ is a probability measure, $\|h\|_\infty\ge\|h\|_2$. Conversely, Cauchy--Schwarz and the addition theorem
\[
 \sum_{r=1}^{h_\ell^{(d)}}Y_{\ell,r}(v)^2=h_\ell^{(d)}
\]
give $\|h\|_\infty\le\sqrt{M_T}\,\|h\|_2$. Apply \cref{thm:sharp-defect}.
\end{proof}

\begin{definition}[Kernel-adapted design]\label{def:kernel-adapted}
A normalized finite measure $\mu$ satisfying \eqref{eq:main-iv} is called a \emph{$\psi$-adapted directional design}.  For equal weights we call its support a \emph{kernel-adapted design}.
\end{definition}

Thus \cref{thm:main-criterion} may be summarized as
\begin{equation}\label{eq:main-slogan}
 \boxed{\quad
 \text{universal exact integral geometry for }\psi
 \iff
 \text{design exactness on }\Lambda_\psi.
 \quad}
\end{equation}

This formulation makes clear that a full spherical design is often stronger than necessary.

\begin{corollary}[Spherical designs are universal finite-spectrum rules]\label{cor:spherical-sufficient}
If $\Lambda_\psi\subset\{1,\ldots,t\}$, then every weighted spherical $t$-design is universally exact for $\psi$.
\end{corollary}

\begin{corollary}[Finite-spectrum existence]\label{cor:caratheodory}
Suppose $\Lambda_\psi$ is finite and put
\[
 M_\psi:=\sum_{\ell\in\Lambda_\psi}h_\ell^{(d)}.
\]
Then there exists a positive normalized finite rule
\[
 \mu=\sum_{j=1}^Nw_j\delta_{x_j},\qquad w_j>0,
\]
that is universally exact for $\psi$, with
\begin{equation}\label{eq:caratheodory-bound}
 N\le M_\psi+1.
\end{equation}
\end{corollary}

\begin{proof}
Choose an orthonormal basis of each $\Harm_\ell$, $\ell\in\Lambda_\psi$, and collect all basis functions into a continuous map
\[
 \Phi:\Sph^{d-1}\to\R^{M_\psi}.
\]
Since every nonconstant spherical harmonic has zero $\sigma$-mean,
\[
 0=\int\Phi(u)\,d\sigma(u)
\]
belongs to the convex hull of $\Phi(\Sph^{d-1})$.  Carath\'eodory's theorem gives points $x_j$ and positive weights $w_j$ summing to one, with $N\le M_\psi+1$, such that
\[
 \sum_jw_j\Phi(x_j)=0.
\]
Thus $\mu$ annihilates every active harmonic space, and \cref{thm:main-criterion} applies.
\end{proof}

\begin{remark}
The bound \eqref{eq:caratheodory-bound} is intentionally elementary and generally far from optimal.  The design and cubature literature provides sharper lower and upper bounds in many projective and harmonic-index settings; see \cite{BOT,ZhuEtAl,Lyubich2009}.  Its role here is conceptual: finite spectral support alone guarantees a positive finite exact realization.
\end{remark}

\begin{example}[A sparse kernel needs only a sparse design]\label{ex:sparse}
Let $Z_{2r}$ be any nonzero zonal harmonic kernel of degree $2r$ and set
\[
 \psi(s)=a_0+a_{2r}Z_{2r}(s),\qquad a_{2r}\ne0.
\]
Then $\Lambda_\psi=\{2r\}$.  Universal exactness therefore requires only a harmonic-index $2r$ design, not a spherical $2r$-design.  This is the simplest illustration of spectral economy.
\end{example}

\section{Power-Cauchy kernels and a complete classification}

For $p\ge0$, define
\[
 \psi_p(s)=|s|^p,
 \qquad
 \cB_pK(u):=\cB_{\psi_p}K(u)
 =\frac12\int_{\Sph^{d-1}}|\langle u,v\rangle|^p\,dS_K(v).
\]
For $p=1$, $\cB_1K$ is the classical brightness function.  For $p=2m$, the absolute value disappears and the kernel becomes polynomial.

The zeroth multiplier is the familiar spherical moment
\begin{equation}\label{eq:cdp}
 c_{d,p}:=\lambda_0(\psi_p)
 =\frac{\Gamma(d/2)\Gamma((p+1)/2)}{\sqrt\pi\,\Gamma((d+p)/2)}.
\end{equation}

\begin{proposition}[Funk--Hecke multipliers of $|s|^p$]\label{prop:power-multipliers}
Let $p>-1$.  For every $n\ge0$,
\begin{equation}\label{eq:power-multipliers}
 \lambda_{2n}(\psi_p)
 =c_{d,p}
 \frac{(-1)^n(-p/2)_n}{((d+p)/2)_n},
 \qquad
 \lambda_{2n+1}(\psi_p)=0.
\end{equation}
Here $(a)_n=\Gamma(a+n)/\Gamma(a)$ is the Pochhammer symbol.
\end{proposition}

\begin{proof}
Odd multipliers vanish because $\psi_p$ is even.  We give the Gegenbauer calculation for $d\ge3$ and note that the same formula in $d=2$ follows by the corresponding Fourier cosine integral.

Put $\alpha=(d-2)/2$.  With normalized spherical measure, the Funk--Hecke formula gives
\begin{equation}\label{eq:FH-gegen}
 \lambda_{2n}(\psi_p)
 =\frac{1}{B(1/2,\alpha+1/2)\,C_{2n}^{\alpha}(1)}
 \int_{-1}^1|s|^pC_{2n}^{\alpha}(s)(1-s^2)^{\alpha-1/2}\,ds.
\end{equation}
The quadratic change of variables $y=s^2$ and the identity
\begin{equation}\label{eq:gegen-jacobi}
 C_{2n}^{\alpha}(\sqrt y)
 =\frac{(\alpha)_n}{(1/2)_n}
 P_n^{(\alpha-1/2,-1/2)}(2y-1)
\end{equation}
reduce the integral to a beta--hypergeometric integral.  Using
\[
 P_n^{(\alpha-1/2,-1/2)}(2y-1)
 =(-1)^n\frac{(1/2)_n}{n!}
 {}_2F_1\!\left(\begin{matrix}-n,n+\alpha\\[1mm]1/2\end{matrix};y\right)
\]
and integrating the terminating series termwise yields a balanced ${}_3F_2(1)$.  Saalsch\"utz's summation gives
\begin{align}\label{eq:gegen-integral}
 &\int_{-1}^1|s|^pC_{2n}^{\alpha}(s)(1-s^2)^{\alpha-1/2}\,ds\\
 &\quad=
 B\!\left(\frac{p+1}{2},\alpha+\frac12\right)
 \frac{(-1)^n(\alpha)_n(\alpha+1/2)_n(-p/2)_n}
 {n!(1/2)_n(\alpha+p/2+1)_n}.
\end{align}
Finally,
\[
 C_{2n}^{\alpha}(1)
 =\frac{(\alpha)_n(\alpha+1/2)_n}{n!(1/2)_n},
\]
so substitution into \eqref{eq:FH-gegen}, together with
\[
 \alpha+\frac p2+1=\frac{d+p}{2},
\]
gives \eqref{eq:power-multipliers}.  The quotient of beta functions for $n=0$ is exactly \eqref{eq:cdp}.
\end{proof}

\begin{corollary}[Active spectrum of power kernels]\label{cor:power-spectrum}
Let $p\ge0$.
\begin{enumerate}
\item If $p=2m$ with $m\in\N$, then
\begin{equation}\label{eq:active-even-power}
 \Lambda_{\psi_{2m}}=\{2,4,\ldots,2m\}.
\end{equation}
\item If $p>0$ and $p\notin2\N$, then
\begin{equation}\label{eq:active-noneven-power}
 \Lambda_{\psi_p}=\{2,4,6,\ldots\}.
\end{equation}
\item For $p=0$, the active spectrum is empty.
\end{enumerate}
\end{corollary}

\begin{proof}
If $p=2m$, then $(-p/2)_n=(-m)_n$ is nonzero for $1\le n\le m$ and vanishes for $n>m$.  If $p>0$ is not an even integer, $-p/2$ is not a nonpositive integer, so $(-p/2)_n\ne0$ for every $n\ge1$.  The case $p=0$ is constant.
\end{proof}

We now obtain the exactness classification.

\begin{theorem}[Power-Cauchy exactness classification]\label{thm:power-classification}
Let $\mu$ be a finite signed normalized directional rule on $\Sph^{d-1}$.
\begin{enumerate}
\item If $p=2m$, $m\ge1$, then
\begin{equation}\label{eq:power-exact-iff}
 \int\cB_{2m}K\,d\mu
 =\int\cB_{2m}K\,d\sigma
 \quad\forall K\in\K^d_\circ
\end{equation}
holds if and only if
\begin{equation}\label{eq:projective-moments}
 \int Y\,d\mu=0
 \qquad\forall Y\in\Harm_{2r},\quad 1\le r\le m.
\end{equation}
Equivalently, the antipodal symmetrization of $\mu$ is a weighted real projective $m$-design.
\item If $p>0$ and $p\notin2\N$, then no finite signed normalized directional rule can satisfy universal exactness.
\end{enumerate}
\end{theorem}

\begin{proof}
Part (i) is immediate from \cref{thm:main-criterion,cor:power-spectrum}.

For part (ii), universal exactness would imply
\[
 \int Y\,d\mu=0
 \qquad\forall Y\in\Harm_{2r},\quad r\ge1.
\]
Let
\[
 \mu^{\sym}:=\frac12(\mu+(-\Id)_\#\mu).
\]
Then $\mu^{\sym}$ has total mass one, annihilates every nonconstant even spherical harmonic, and automatically annihilates every odd harmonic.  Hence $\mu^{\sym}$ and $\sigma$ agree on all spherical polynomials.  By density of spherical polynomials in $C(\Sph^{d-1})$, they are the same finite signed measure:
\[
 \mu^{\sym}=\sigma.
\]
The left side is finite atomic, while $\sigma$ is non-atomic, a contradiction.
\end{proof}

\begin{remark}[Relation to classical projective cubature]
The finite-dimensional cubature condition in \cref{thm:power-classification}(i) is part of the established theory connecting projective designs, spherical designs, and isometric embeddings into finite-dimensional $\ell_p$ spaces; see \cite{LyubichVaserstein,Lyubich2009}.  The contribution here is the universal convex-body interpretation supplied by \cref{thm:main-criterion}: the same projective moment condition is exactly what is needed to replace the continuous directional integral for the whole class of surface area measures.
\end{remark}

Combining \cref{prop:continuous-average} with \eqref{eq:cdp} gives explicit surface-area formulas.

\begin{corollary}[Exact finite power-Cauchy surface-area formula]\label{cor:surface-power}
Let $m\ge1$, and let $\mu=\sum_{j=1}^Nw_j\delta_{x_j}$ satisfy \eqref{eq:projective-moments}.  Then for every $K\in\K^d_\circ$,
\begin{equation}\label{eq:surface-power}
 \boxed{
 S(K)
 =\frac{2(d/2)_m}{(1/2)_m}
 \sum_{j=1}^Nw_j\cB_{2m}K(x_j).
 }
\end{equation}
In particular, every spherical $2m$-design gives an equal-weight exact rule.
\end{corollary}

\begin{proof}
For $p=2m$,
\[
 c_{d,2m}=\frac{(1/2)_m}{(d/2)_m}.
\]
Use universal exactness and \cref{prop:continuous-average}.
\end{proof}

\section{The quadratic case and node economy}

The case $m=1$ gives the simplest illustration of spectral economy. The exactness condition reduces to the familiar isotropic-frame identity, from which the minimal positive rule follows immediately.

\begin{proposition}[Quadratic exactness and isotropic frames]\label{prop:quad-frame}
Let
\[
 \mu=\sum_{j=1}^Nw_j\delta_{x_j},\qquad \sum_jw_j=1.
\]
Then universal exactness for $\psi_2(s)=s^2$ is equivalent to
\begin{equation}\label{eq:tight-frame}
 \sum_{j=1}^Nw_jx_jx_j^\top=\frac1dI_d.
\end{equation}
\end{proposition}

\begin{proof}
A homogeneous quadratic polynomial decomposes into a constant multiple of $|x|^2$ and a harmonic quadratic.  Thus annihilation of $\Harm_2$ is equivalent to equality of the second moment matrix with the rotationally invariant second moment
\[
 \int_{\Sph^{d-1}}uu^\top\,d\sigma(u)=\frac1dI_d.
\]
Apply \cref{thm:power-classification}.
\end{proof}

\begin{theorem}[Minimal positive quadratic rule]\label{thm:minimal-quadratic}
Suppose $w_j>0$ and \eqref{eq:tight-frame} holds.  Then $N\ge d$.  If $N=d$, then
\[
 w_1=\cdots=w_d=\frac1d
\]
and $x_1,\ldots,x_d$ form an orthonormal basis, up to signs.  Conversely, every orthonormal basis with weights $1/d$ gives an exact rule.
\end{theorem}

\begin{proof}
The matrix on the right side of \eqref{eq:tight-frame} has rank $d$, whereas the left side is a sum of $N$ rank-one matrices, so $N\ge d$.

Assume $N=d$ and set
\[
 A=\begin{bmatrix}\sqrt{dw_1}x_1&\cdots&\sqrt{dw_d}x_d\end{bmatrix}.
\]
Then \eqref{eq:tight-frame} says $AA^\top=I_d$.  Since $A$ is square, it is orthogonal.  Hence each column has norm one, so $dw_j=1$ for every $j$, and the vectors $x_j$ are orthonormal.  The converse is immediate.
\end{proof}

\begin{corollary}[An exact $d$-direction surface-area identity]\label{cor:d-directions}
Let $e_1,\ldots,e_d$ be an orthonormal basis of $\R^d$.  Then for every convex body $K\in\K^d_\circ$,
\begin{equation}\label{eq:d-direction-formula}
 \boxed{
 S(K)=2\sum_{j=1}^d\cB_2K(e_j).
 }
\end{equation}
Among positive weighted rules, $d$ unoriented directions are minimal.
\end{corollary}

\begin{proof}
Since $c_{d,2}=1/d$, \cref{cor:surface-power} with weights $1/d$ gives \eqref{eq:d-direction-formula}.  Minimality follows from \cref{thm:minimal-quadratic}.
\end{proof}

\begin{example}[Three dimensions]
For $d=3$ and any mutually orthogonal unit vectors $e_1,e_2,e_3$,
\[
 S(K)=2\bigl(\cB_2K(e_1)+\cB_2K(e_2)+\cB_2K(e_3)\bigr).
\]
Thus three unoriented directions recover the surface area of every convex body exactly from quadratic normal moments.  The functional $\cB_2$ should not be confused with ordinary projected area: it is the quadratic analogue of the Cauchy brightness functional.
\end{example}

\section{Classical Cauchy brightness: impossibility and controlled approximation}

For $\psi(s)=|s|$, the Cauchy projection formula gives
\begin{equation}\label{eq:cauchy-brightness}
 \cB_1K(u)=\operatorname{vol}_{d-1}(K|u^\perp).
\end{equation}
The continuous average is
\begin{equation}\label{eq:cauchy-average}
 \int_{\Sph^{d-1}}\cB_1K(u)\,d\sigma(u)
 =\frac{c_{d,1}}2S(K),
 \qquad
 c_{d,1}=\frac{\Gamma(d/2)}{\sqrt\pi\,\Gamma((d+1)/2)}.
\end{equation}
By \cref{thm:power-classification}, no finite signed rule reproduces \eqref{eq:cauchy-average} for every convex body.  We now quantify what a spherical design does achieve.

For $\psi\in C([-1,1])$ define the best uniform polynomial approximation error
\begin{equation}\label{eq:Et}
 E_t(\psi)_\infty
 :=\inf_{q\in\mathcal P_t([-1,1])}
 \|\psi-q\|_{L^\infty([-1,1])}.
\end{equation}

\begin{theorem}[Uniform design approximation]\label{thm:uniform-approx}
Let $X_N\subset\Sph^{d-1}$ be a spherical $t$-design, and write
\[
 Q_XF:=\frac1N\sum_{x\in X_N}F(x).
\]
Then for every even $\psi\in C([-1,1])$ and every $K\in\K^d_\circ$,
\begin{equation}\label{eq:uniform-approx}
 \left|
 Q_X(\cB_\psi K)
 -\int_{\Sph^{d-1}}\cB_\psi K\,d\sigma
 \right|
 \le E_t(\psi)_\infty\,S(K).
\end{equation}
\end{theorem}

\begin{proof}
Fix $v\in\Sph^{d-1}$ and let $q$ be any polynomial of degree at most $t$.  The function $u\mapsto q(\langle u,v\rangle)$ is a spherical polynomial of degree at most $t$, hence the design is exact on it.  Therefore
\[
 \left|
 \frac1N\sum_{x\in X_N}\psi(\langle x,v\rangle)
 -\int\psi(\langle u,v\rangle)\,d\sigma(u)
 \right|
 \le2\|\psi-q\|_\infty.
\]
Multiply by $1/2$, integrate with respect to $S_K(v)$, and take the infimum over $q$.
\end{proof}

If $\lambda_0(\psi)\ne0$, define the design-based surface-area estimator
\begin{equation}\label{eq:surface-estimator-general}
 \widehat S_{\psi,X}(K)
 :=\frac{2}{\lambda_0(\psi)}Q_X(\cB_\psi K).
\end{equation}
Then \cref{prop:continuous-average,thm:uniform-approx} imply
\begin{equation}\label{eq:relative-general}
 \frac{|\widehat S_{\psi,X}(K)-S(K)|}{S(K)}
 \le\frac{2E_t(\psi)_\infty}{|\lambda_0(\psi)|}.
\end{equation}

\begin{corollary}[Uniform approximation of the Cauchy surface-area formula]\label{cor:cauchy-approx}
Let $X_N$ be a spherical $t$-design and define
\begin{equation}\label{eq:cauchy-estimator}
 \widehat S_X(K)
 :=\frac{2}{c_{d,1}}\frac1N\sum_{x\in X_N}\operatorname{vol}_{d-1}(K|x^\perp).
\end{equation}
Then
\begin{equation}\label{eq:cauchy-relative}
 \boxed{
 \frac{|\widehat S_X(K)-S(K)|}{S(K)}
 \le\frac{2E_t(|\cdot|)_\infty}{c_{d,1}}
 }
\end{equation}
for every convex body $K$.
In particular,
\begin{equation}\label{eq:cauchy-Ot}
 \sup_{K\in\K^d_\circ}
 \frac{|\widehat S_X(K)-S(K)|}{S(K)}
 =O_d(t^{-1}).
\end{equation}
\end{corollary}

\begin{proof}
Only the last statement needs comment.  Jackson's theorem gives $E_t(|\cdot|)_\infty\le C/t$; see, e.g., \cite{DeVoreLorentz}.  Substitute into \eqref{eq:cauchy-relative}.
\end{proof}

\begin{remark}[Spectral meaning]
For $|s|$, every positive even harmonic multiplier is nonzero.  A spherical $t$-design removes all active modes of degree at most $t$ exactly; the approximation error is therefore the unresolved high-frequency tail.  This is the precise sense in which exact finite integral geometry passes continuously into spectral approximation beyond the exactness threshold.
\end{remark}

\section{Exact projection moments for submanifolds}

The same finite-isotropy mechanism is not restricted to convex bodies.  Let $M\subset\R^d$ be a compact countably $k$-rectifiable set with finite $\mathcal H^k$ measure, $1\le k\le d-1$, and let $T_xM$ denote the approximate tangent $k$-plane for $\mathcal H^k$-a.e. $x$.

For $u\in\Sph^{d-1}$ let $P_{u^\perp}$ be orthogonal projection onto $u^\perp$ and let
\[
 J_k(P_{u^\perp}|_{T_xM})
\]
be its $k$-dimensional Jacobian.  Define
\begin{equation}\label{eq:Jkm}
 \mathcal J_{k,m}(M;u)
 :=\int_M J_k(P_{u^\perp}|_{T_xM})^{2m}\,d\mathcal H^k(x).
\end{equation}

\begin{theorem}[Exact finite projection-moment identity]\label{thm:projection-moment}
Let $m\ge1$, and let $\mu=\sum_jw_j\delta_{x_j}$ be any normalized rule annihilating $\Harm_2,\ldots,\Harm_{2m}$.  Then
\begin{equation}\label{eq:projection-moment}
 \boxed{
 \sum_jw_j\mathcal J_{k,m}(M;x_j)
 =\frac{((d-k)/2)_m}{(d/2)_m}\,\mathcal H^k(M).
 }
\end{equation}
In particular, every spherical $2m$-design and every weighted real projective $m$-design satisfies \eqref{eq:projection-moment}.
\end{theorem}

\begin{proof}
Fix a $k$-plane $E\subset\R^d$ and choose an orthonormal basis $e_1,\ldots,e_k$.  The Gram matrix of the projected vectors $P_{u^\perp}e_a$ is
\[
 G=I_k-zz^\top,
 \qquad
 z=(\langle e_1,u\rangle,\ldots,\langle e_k,u\rangle)^\top.
\]
The matrix determinant lemma gives
\[
 J_k(P_{u^\perp}|_E)^2
 =\det(I_k-zz^\top)
 =1-\|P_Eu\|^2.
\]
Hence
\[
 J_k(P_{u^\perp}|_E)^{2m}
 =(1-\|P_Eu\|^2)^m,
\]
a spherical polynomial of degree at most $2m$ containing only even harmonic degrees.  Our moment assumptions therefore give exact integration over $\mu$.

For $U$ uniform on $\Sph^{d-1}$,
\[
 \|P_EU\|^2\sim\operatorname{Beta}\!\left(\frac k2,\frac{d-k}{2}\right),
\]
so
\[
 \int_{\Sph^{d-1}}(1-\|P_Eu\|^2)^m\,d\sigma(u)
 =\frac{((d-k)/2)_m}{(d/2)_m}.
\]
The constant is independent of $E$.  Apply the pointwise identity to $E=T_xM$ and integrate over $M$.
\end{proof}

\begin{corollary}[Hypersurface version]\label{cor:hypersurface}
If $M$ is a rectifiable hypersurface with unit normal $n(x)$ defined a.e., then
\begin{equation}\label{eq:hypersurface}
 \sum_jw_j\int_M|\langle x_j,n(x)\rangle|^{2m}\,d\mathcal H^{d-1}(x)
 =\frac{(1/2)_m}{(d/2)_m}\,\mathcal H^{d-1}(M).
\end{equation}
\end{corollary}

\begin{proof}
For a hypersurface, the $(d-1)$-Jacobian of projection onto $u^\perp$ is $|\langle u,n(x)\rangle|$.  Set $k=d-1$ in \cref{thm:projection-moment}.
\end{proof}

For $M=\partial K$, \eqref{eq:hypersurface} is precisely the geometric content behind the power-Cauchy formula \eqref{eq:surface-power}.

\section{Discussion and outlook}

The main conclusion of the paper is that exact finite directional integral geometry is governed by the nullspace and spectral support of a zonal convolution operator. The sharp identity \eqref{eq:sharp-defect} reduces the worst-case error over all convex bodies to an $L^\infty$ potential discrepancy, while \cref{cor:l2-discrepancy} links that geometric defect to a weighted $L^2$ harmonic discrepancy when the active spectrum is finite. Funk--Hecke diagonalization then converts vanishing of the discrepancy into the kernel-adapted moment conditions in \cref{thm:main-criterion}. In particular, a spherical design is a universal device for finite-spectrum kernels, but it can impose substantially more conditions than necessary.

The evenness assumption has a geometric meaning: unoriented directions form the projective space
\[
 \RP^{d-1}=\Sph^{d-1}/\{u\sim-u\}.
\]
For the power kernel $|\langle u,v\rangle|^{2m}$, the directional observable has projective polynomial degree $m$, and \cref{thm:power-classification} identifies universal exactness precisely with weighted projective $m$-design exactness. By contrast, non-even powers have full positive even spectrum and therefore cannot be represented by any finite atomic directional rule. The classical Cauchy brightness kernel $|s|$ is the basic example: exact finite realization fails, while \cref{cor:cauchy-approx} quantifies how polynomial approximation and spherical designs recover the continuous formula asymptotically.

This perspective is complementary to the classical theory of multiplier transforms and geometric tomography \cite{Gardner,Kiderlen,GYY}, to the theory of harmonic-index and projective designs \cite{BOT,ZhuEtAl,Lyubich2009,Lindblad}, and to work on $L_p$ cosine transforms and isotropic measures \cite{LiXiZhang}. The emphasis here is the exact finite replacement of an invariant directional average for the entire class of convex bodies, together with the sharp defect formula that makes the convex-geometric quantifier ``for every $K$'' completely explicit.

Santal\'o's kinematic formulas live on larger homogeneous spaces, and the same question can be posed there: which invariant Crofton, Kubota or kinematic averages admit exact finite realization? Grassmannian designs suggest the appropriate algebraic objects for projection and section formulas, while designs or cubatures on compact groups should enter genuine kinematic formulas. The present directional theory isolates the first layer of this broader program, which we refer to as \emph{Exact Finite Integral Geometry}.

\section*{Data availability}
No datasets were generated or analysed in this study. All mathematical results and proofs are contained in the article.

\end{document}